\documentclass[11pt]{amsart}
\usepackage{amsmath,amsthm,amssymb, amsfonts,hyperref, amscd}

\newtheorem{thm}{Theorem}[section]
\newtheorem{cor}[thm]{Corollary}
\newtheorem{lem}[thm]{Lemma}
\newtheorem{prop}[thm]{Proposition}
\newtheorem{conj}[thm]{Conjecture}
\newtheorem*{thm1.3}{Theorem 1.3$^{\,\prime}$}

\theoremstyle{plain}
\numberwithin{equation}{thm}
\theoremstyle{remark}

\def\P{\mathbb P}
\def\Q{\mathbb Q}

\def\lra{\longrightarrow}
\def\forb{\mathcal{O}_{\phi}^{+}}
\def\orb{\mathcal{O}_{\phi}^{-}}

\def\m{\mu_{\phi, \beta}(n)}
\def\phi{\varphi}

\title{Integer Points in Backward Orbits}
\author{Vijay A. Sookdeo}

\address{
Vijay Sookdeo\\
Department of Mathematics\\
The Catholic University of America\\
Washington, DC 20064 \\
}

\email{sookdeo@cua.edu}

\begin{document}

\begin{abstract}
A theorem of J. Silverman states that a forward orbit of a rational map $\phi(z)$ on $\mathbb P^1(K)$ contains finitely many $S$-integers in the number field $K$ when $(\phi\circ\phi)(z)$ is not a polynomial.  We state an analogous conjecture for the backward orbits using a general $S$-integrality notion based on the Galois conjugates of points.  This conjecture is proven for the map $\phi(z)=z^d$, and consequently Chebyshev polynomials, by uniformly bounding the number of Galois orbits for $z^n-\beta$ when $\beta\not =0$ is a non-root of unity.  In general, our conjecture is true provided that the number of Galois orbits for $\phi^n(z)-\beta$ is bounded independently of $n$.
\end{abstract}

\maketitle

\section{Introduction}

Let $K$ be a number field, $\phi:\mathbb P^1 \lra \mathbb P^1$ be a rational map of degree $\ge 2$ defined over $K$, and $\phi^n(z)$ be the $n$th iterate $(\phi\circ \cdots \circ\phi)(z)$.  The \emph{forward orbit} of $\beta \in \mathbb P^1(K)$ under $\phi$ is defined as $\forb(\beta)=\{\beta, \phi(\beta), \phi^2(\beta), \dots \}$ and the \emph{backward orbit} is defined as the collection of inverse images $$\orb(\beta)=\bigcup_{n\ge 0}\phi^{-n}(\beta).$$ A point $\beta$ is \emph{preperiodic} for $\phi$ if $\forb(\beta)$ is finite and \emph{exceptional} for $\phi$ if $\orb(\beta)$ is finite.  We write $\mbox{PrePer}(\phi, \overline K)$ for the set of preperiodic points of $\phi$ in $\mathbb P^1(\overline K)$, and $S$ for a finite set of places of $K$ which includes all the archimedean places.

When $\phi$, or some iterate of $\phi$, is a polynomial, $\forb$ may contain infinitely many distinct points in $\mathcal O_{K, S}$, the ring of $S$-integers in $K$.  In 1993, Silverman \cite{sil} proved that if $\phi^2(z)$ is not a polynomial, then $\forb(\beta)$ contains at most finitely many points in $\mathcal O_{K, S}$.  We would like to state an analogous conjecture for the backward orbit $\orb$ and give some evidences to support it. Since $\orb(\beta)\cap\P^1(L)$ is finite for any $\beta$ and any number field $L$ (see Corollary \ref{fin}), it is trivial to ask when $\orb$ contains finitely many points in $\mathcal O_{K,S}$.  It is better to ask is what conditions will guarantee $\orb(\beta)$ contains at most finitely many points in $\mathcal O_{\overline K, S}$, the ring of $S$-integers in $\overline K$.

To formulate the conjecture for backward orbits, we restate Silverman's result using a more flexible, geometric notion of integrality.  The set $\mathcal O_{K,S}$ can be thought of as all the points $P=[\gamma:1]\in \P^1(K)$ whose $v$-adic chordal distance $\delta_v(P, \infty)=1$ for all $v\not\in S$ (see Section \ref{integrality}). This means that the $S$-integral points in $K$ can be defined relative to $\infty$.  Additionally, the condition that $\phi^2(z)\not\in K[z]$ is equivalent to $\infty$ not being exceptional for $\phi$ (see \cite{sil}).  Therefore, we may state Silverman's result as follows:  If $\infty$ is not exceptional for $\phi$, then $\forb(\beta)$ contains at most finitely many points in $\P^1(K)$ which are $S$-integral relative to $\infty$. Supposing $f$ is a coordinate change of $\P^1$ taking $\infty$ to $\alpha$, we have, after possibly enlarging $S$, that $\gamma$ is $S$-integral relative to $\infty$ if and only if $f(\gamma)$ is $S$-integral relative to $\alpha$ (see Section \ref{integrality}).  This gives the following version of Silverman's Theorem.

\begin{thm}[Silverman]\label{sil}
If $\alpha \in \mathbb P^1(K)$ is not exceptional for $\phi$, then $\forb(\beta)$ contains at most finitely many points in $\mathbb P^1(K)$ which are $S$-integral relative to $\alpha$.
\end{thm}


Since an exceptional point in analogous to a preperiodic point, Silverman's Theorem motivates the following conjecture for backward orbits.

\begin{conj}\label{conj}
If $\alpha \in \mathbb P^1(K)$ is not preperiodic for $\phi$, then $\orb(\beta)$ contains at most finitely many points in $\P^1(\overline K)$ which are $S$-integral relative to $\alpha$.
\end{conj}

The main theorem of this paper is the following which immediately gives Conjecture \ref{conj} for the map $\phi(z)=z^d$.  By the functorial properties of relative $S$-integrally, Conjecture \ref{conj} will also be true for Chebyshev polynomials.

\begin{thm}\label{main-thm}
Suppose $\alpha \in K$ is not 0 or a root of unity.  Then there are at most finitely many points in $\{ \gamma \in \overline{K} \mid \gamma^n=\beta \}$ which are $S$-integral relative to $\alpha$.
\end{thm}

It is important to note that unlike Theorem \ref{sil}, Conjecture \ref{conj} and Theorem \ref{main-thm} are integrality statements over $\overline{K}$. The definition for when $\gamma \in \overline{K}$ is $S$-integral relative to $\alpha \in K$ must not depend on how $\gamma$ embeds into $\overline{K}$.  Therefore, it will useful to know something about the Galois orbits for $\phi^n(z)-\beta$ in proving our conjecture.  In fact, Theorem 1.3$^{\,\prime}$ below, which is an immediate consequence of Lemma \ref{lang}, bounds the number of Galois orbits for $z^n-\beta$ and this is enough to give Theorem \ref{main-thm}.

\begin{thm1.3}  Suppose $\beta \in K$ is not 0 or a root of unity.  Then the number of Galois orbits for $z^n-\beta$ is bounded by a constant independent of $n$.
\end{thm1.3}


To see how to utilize the connection between Galois orbits and relative $S$-integrality in $\overline{K}$, suppose the points in $\phi^{-n}(\beta)$ are all Galois conjugates for each $n$.  Then the projection formula (see Section \ref{integrality}) translates Conjecture \ref{conj} into a statement about forward orbits.  This will consequently give a proof via Silverman's Theorem.  Since we cannot expect all the points in $\phi^{-n}(\beta)$ to be Galois conjugates, a more plausible hypothesis is considered in Theorem \ref{cor2}.  One way this hypothesis can be satisfied is to show that the number of Galois orbits for $\phi^{-n}(\beta)$ is bounded by a constant independent of $n$.  This sort of bound was established by R. Jones \cite{rafe} for certain types of quadratic polynomials with $\beta = 0$.  More generally, it is shown in Section \ref{dl} that when $\beta$ is not preperiodic for $\phi$, the Dynamical Lehmer's Conjecture implies such a bound on the number of Galois orbits for $\phi^{-n}(\beta)$, and therefore gives Conjecture \ref{conj} in this case.



There are strong similarities between $\orb(\beta)$ and $\mbox{PrePer}(\phi, \overline K)$.  It has been similarly conjectured that $\mbox{PrePer}(\phi, \overline K)$ contains finitely many points which are $S$-integral relative to a non-preperiodic point $\alpha$ of $\phi$.  This conjecture of S. Ih has been proven for the map $\phi(z)=z^d$ with $d\ge2$ by Baker-Ih-Rumely \cite{bir}.   More recently, C. Petsche \cite{clay} has shown Ih's Conjecture is true when the non-preperiodic point $\alpha$ is totally Fatou for $\phi$.

Both sets also share similar equidistribution properties. Lyubich \cite{lyu} has shown that the points in $\orb(\beta)$ and $\mbox{PrePer}(\phi)$ are equidistributed with respect to the Haar measure on $\mathbb P^1(\mathbb C)$.  Later, Baker-Rumely \cite{br} and C. Favre and J. Rivera-Letelier \cite{FRL} extended Lyubich's result to any set of points $P_n\in \mathbb P^1(\overline{K})$ with $\hat{h}_{\phi}(P_n) \lra 0$.  Chambert-Loir \cite{CL} has also proven analogous equidistribution results for such sequences points on certain elliptic curves.  However, one cannot expect to have integrality results for general families of points with canonical height tending to zero. For example, let $K =\mathbb Q$, $\phi(z)=z^2$, $S=\{\infty\}$ and $\alpha = 2$.  In \cite{bir}, it was shown that if $\beta_n$ is a root of the polynomial $f_n(z)=z^{2^{n}}(z-2)-1$, then $\hat{h}_{\phi}(\beta_n) \lra 0$ and each $\beta_n$ is $S$-integral relative to $\alpha$.

\emph{Acknowledgement.}  The author would like to thank T. Tucker, M. Zieve, and the referee for useful comments, suggestions, and corrections.

\section{Height, Relative $S$-Integrality, and Preliminary Results}

\subsection{Heights}{\label{heights}}
Let $M_{\mathbb Q}$ be the set consisting of the usual archimedean absolute value on $\mathbb Q$, along with the $p$-adic absolute values normalized so that $|p|_p=1/p$.  For a number field $K$, $M_K$ will denote the set of normalized inequivalent absolute values constructed from $M_{\mathbb Q}$ in the following manner:  Write $K_v$ for the completion of $K$ at the place $v$ and define $$|\alpha|_v = |N_{K_v/\mathbb Q_p}(\alpha)|_p^{1/[K:\mathbb Q]}$$ for $\alpha \in K$ and the place $v$ lying over $p$.  This normalization gives the product formula $$\prod_{v\in M_K}|\alpha|_v = 1.$$

For $\beta = (\beta_1: \beta_2) \in \mathbb P^1(\overline{K})$, where $\beta_1, \beta_2 \in L$, we define the \emph{absolute logarithmic height} as $$h(\beta) =  \sum_{v \in M_L} \log{\max \{|\beta_1|_v, |\beta_2|_v\}}.$$ This definition is independent of the choice of the field $L$ containing $\beta_1$ and $\beta_2$, and by the product formula, it is also independent of the choice of projective coordinates for $\beta$.  If $\beta = \beta_1/\beta_2 \in \mathbb Q$ with $\beta_1$ and $\beta_2$ relatively prime, then $h(\beta) = \log{\max\{|\beta_1|,|\beta_2|\}}$ and can be used to bound the maximum number of digits needed to write $\beta$.  Therefore,  one may think of the height as measuring the ``arithmetic complexity" of an algebraic number.

An often useful property of the logarithmic height is given by Northcott's Theorem.
\begin{thm}[Northcott]\label{Northcott}
Any set of points of bounded height and bounded degree in $\mathbb P^1(\overline K)$ is finite.
\end{thm}

\begin{proof}
See \cite[Th. 1.6.8.]{hdg}
\end{proof}

Northcott's Theorem implies that $\orb(\beta)$ will contain finitely many points in any fixed number field $L$.  When $\phi$ is a polynomial, this means that the irreducible factors of $\phi^n(z)-\beta$ over $K$ will have degrees growing larger with $n$.  More generally, it will be shown in Section \ref{dl} that the Dynamical Lehmer's Conjecture implies that the number of irreducible factors is bounded by a constant independent of $n$.

\begin{cor}\label{fin}
For any $\beta \in \mathbb P^1(K)$, $\orb(\beta)$ contains finitely many points in any finite extension $L$ of $\Q$.
\end{cor}

\begin{proof}
If $\gamma \in \orb(\beta)$, then $\phi^n(\gamma)=\beta$ for some $n$.  By functoriality, we have $h(\phi(\alpha)) = dh(\alpha)+O(1)$ \cite[Th. 3.11]{ADS}.  This gives $d^n h(\gamma) + O(1+d+ \dots +d^{n-1}) = h(\beta) $ which implies $h(\gamma)$ is bounded.  So $\orb(\beta) \cap L$ is a set of bounded height and degree, and therefore finite by Northcott's Theorem.
\end{proof}

\subsection{$S$-integrality}{\label{integrality}}  Let $S$ be a finite set of places of $K$ containing all the archimedean places, and define the $v$-adic chordal metric on $\mathbb P^1{(\mathbb C_v)}$ as $$\delta_v(P, Q) = \frac{|x_1y_2 - y_1x_2|_v}{\max\{|x_1|_v, |y_1|_v\}  \max\{|x_2|_v, |y_2|_v\}}$$ where $P=[x_1:y_1]$ and $Q=[x_2:y_2]$.  Since $0\le \delta_v(\cdot, \cdot) \le 1$, we can view $\mathcal O_{K,S}$ as the set points $\gamma \in K$ whose $v$-adic chordal distance to $\infty$ is maximal for all $v \not\in S$; that is, $|\gamma|_v\le1$ if and only if $\delta_v(P, \infty)=1$ where $P=[\gamma:1]$.  This geometric view of an $S$-integer allows a generalization to $\overline K$ by allowing $P$ to vary over the embeddings of $\gamma$ in $\overline K$, and by replacing $\infty$ with some arbitrary point.



To state the definition in terms of local heights (see \cite[Ch. 3]{ADS}), define $\lambda_{P,v}(Q) = -\log\delta_v(P, Q)$ and let $\alpha, \beta \in \mathbb P^1(\overline{K})$.  Then we say \emph{$\beta$ is S-integral relative to $\alpha$} if and only if $\lambda_{P,v}(Q) = 0$ for all $v\not\in S$, and for all $P$ and $Q$ varying over the respective $K$-embeddings of $\alpha$ and $\beta$ in $\mathbb P^1(\mathbb C_v)$.  More generally, for the divisor $D=\sum n_iP_i$ on $\mathbb P^1 (\mathbb C_v)$, define $\lambda_{D,v}(Q)= \sum n_i\lambda_{P_i,v}(Q)$ and let $\alpha_i, \beta \in \mathbb P^1(\overline{K})$.  Then $\beta$ is S-integral relative to $\Delta=\sum n_i\alpha_i$ if and only if $\lambda_{D,v}(Q)=0$ for all $v\not\in S$, and for all $P_i$ and $Q$ varying over the respective $K$-embeddings of $\alpha_i$ and $\beta$ in $\mathbb P^1(\mathbb C_v)$.

Restricting to affine coordinates by identifying $\overline{K}$ with the points $[x:1]\in \mathbb P^1(\overline{K})$, our definition becomes: $\beta \in \overline{K}$ is S-integral relative to $\alpha \in \overline{K}$ if and only if, for all $v \not\in S$ and $\sigma, \tau \in  \mbox{Gal}(\overline{K}/K)$,  $$\begin{array}{ll} |\sigma(\beta) - \tau(\alpha)|_v\ge 1 &\mbox{ if $|\tau(\alpha)|_v \le 1$} \\ |\sigma(\beta)|_v\le 1 &\mbox{ if $|\tau(\alpha)|_v > 1$ }.
\end{array} $$
For example, take $K=\Q$ and $S=\{ \infty, 2 \}$.  Then the points $\beta \in \Q$ which are $S$-integral relative to $\infty$ are those point whose denominator may only be divisible by 2.  Similarly, the points $\beta \in \Q$ which are $S$-integral relative to $0$ are those point whose numerator may only be divisible by 2.

After possibly enlarging $S$, this definition is independent of coordinate change on $\mathbb P^1(\overline K)$.  To see this, suppose  $f([X:Y])=[aX+bY:cX+dY]$, with $ad-bc\not=0$, is a linear fractional transformation defined over $K$.  Let $R_v=\{x \in K \mid |x|_v \le 1 \}$ be the valuation ring for $v$ in $K$, and extend $S$ so that $a,b,c,d \in R_v$ and $ad-bc \in R_v^*$ for all $v\not\in S$.  Then \cite[Lem. 2.5]{ADS} implies $\lambda_{f(P),v}(f(Q))=\lambda_{P,v}(Q)$ for all $P,Q\in \mathbb P^1(\overline K)$ and $v\not\in S$.

\subsection{Good Reduction}
Let $v\in M_K$ be a non-archimedean absolute value, $P=[x:y]\in \mathbb P^1(\overline K)$, and $\phi=[F(X,Y):G(X,Y)]$ a rational map defined over $K$ with $f_1,\dots,f_n$ and $g_1,\dots,g_m$ the coefficients of $F(X,Y)$ and $G(X,Y)$, respectively.  We say $P$ and $\phi$ are written in \emph{normalized form} if $\max(|x|_v,|y|_v)=1$ and $\max(|f_1|_v,\dots,|f_n|_v,|g_1|_v,\dots,|g_m|_v)=1$.

Let $R_v=\{x \in K \mid |x|_v \le 1 \}$ be the valuation ring for $v$, $\mathfrak{m}_v=\{x\in K \mid |x|_v=1 \}$ be its maximal ideal, and $\kappa_v =
R_v/\mathfrak{m}_v$ be its residue field.  For $x\in R_v$, we say $\widetilde{x}$, the image of $x$ under the homomorphism $R_v \rightarrow \kappa_v$, is the reduction of $x$ modulo $\mathfrak{m}_v$.  Writing $\phi=[F(X,Y):G(X,Y)]$ in normalized form, we let $\widetilde{\phi}$ be the rational map obtained by reducing the coefficients of $F(X,Y)$ and $G(X,Y)$ modulo $\mathfrak{m}_v$.  The map $\phi$ is said to have \emph{good reduction} at $v$ if $\deg(\phi)=\deg(\widetilde{\phi})$, and \emph{bad reduction} at $v$ otherwise.

Using the Taylor expansion for $\phi(z)=F(z)/G(z)$ around $z=\alpha$, the multiplicity (or ramification) of $\alpha$ at $\phi$ is $e_{\alpha}$ where $\phi(z)-\phi(\alpha)=c(z-\alpha)^{e_{\alpha}} + O((z-\alpha)^{e_{\alpha}+1})$.  For $\phi^{-1}(\beta)=\{\beta_1,\beta_2,\dots,\beta_l \}$, we define the divisor $\phi^*(\beta) = \sum n_i\beta_i$ where $n_i$ is the multiplicity of $\beta_i$ at $\phi(z)-\beta$.

The projection formula, given in the next proposition, tells us that our integrality definition behaves well functorially.  More specifically, if $S$ contains all the places of bad reduction for $\phi$, then $\beta$ is $S$-integral relative to $\phi(\alpha)$ if and only if $\phi^*(\beta)$ is $S$-integral relative to $\alpha$.

\begin{prop}[Projection Formula]\label{proj}  Suppose $v$ is a place of good reduction for $\phi$ and $P,Q \in \mathbb P^1(\overline K) $.  Then $\lambda_{P,v}(\phi(Q))=\lambda_{\phi^*(P),v}(Q)$.
\end{prop}

\begin{proof}
Write $P=[a:b]$, $Q=[x_1:y_1]$, and $\phi=[F(X,Y):G(X,Y)]$ in normalized form. Since $\phi$ has good reduction at $v$, $\phi(Q)=[F(x_1,y_1):G(x_1,y_1)]$ is also in normalized form and $$\delta_v(P,\phi(Q))=|aG(x_1,y_1)-bF(x_1,y_1)|_v.$$

Consider the homogenous polynomial $H(X,Y)=aG(X,Y)-bF(X,Y) \in R'_v[X,Y]$, where $R'_v$ is the ring of integers of $K'$, the splitting field for $H$.  By Gauss's lemma \cite[lem. 1.6.3]{hdg}, we may factor $$H(X,Y)=\prod_{i=1}^d (\beta_iX-\alpha_iY)^{n_i} $$ with $\alpha_i, \beta_i \in R'_v$.  Now
$\widetilde{H}\not=0$ since $\phi$ has good reduction at $v$ and $\max(|a|_v,|b|_v)=1$.  Therefore, $\max(|\alpha_i|_v,|\beta_i|_v)=1$ and the points $P_i=[\alpha_i:\beta_i]$ are written in normalized form.  This gives $\delta_v(P,\phi(Q))=\prod_{i=1}^d \delta_v(P_i,Q)^{n_i}$, and since $H(P_i)=0$
if and only if $P_i\in \phi^{-1}(P)$ with multiplicity $n_i$, taking logarithms give $$ \lambda_{P,v}(\phi(Q))=\sum_{i=1}^d n_i \lambda_{P_i,v}(Q)=\lambda_{\phi^*(P),v}(Q).$$
\end{proof}

\begin{cor}\label{cor1}
Suppose $\phi$ has good reduction for all places $v\not\in S$.  Then $\beta$ is $S$-integral relative to $\phi(\alpha)$ if and only if  $\phi^{*}(\beta)$ is $S$-integral relative to $\alpha$.
\end{cor}

\begin{proof}
If $P$ and $Q$ vary over all the respective embeddings of $\alpha$ and $\beta$ into $\mathbb P^1(\mathbb C_v)$, then $\phi(P)$ and $\phi^*(Q)$ also vary over all the respective embeddings of $\phi(\alpha)$ and $\phi^{*}(\beta)$ into $\mathbb P^1(\mathbb C_v)$.  Since $\phi$ has good reduction for all $v\not\in S$, the projective formula $\lambda_{P,v}(\phi(Q))=\lambda_{\phi^*(P),v}(Q)$ gives the desired result.
\end{proof}

If $S$ is enlarged so that the resultant of $\phi$, $\mbox{Res}(\phi)$, is an $S$-unit, then $\lambda_{P,v}(Q)\le \lambda_{\phi(P),v}(\phi(Q))$ \cite[Th. 2.14]{ADS}.  The would imply a weaker conclusion than Corollary \ref{cor1}:  If $\beta$ is $S$-integral relative to $\alpha$, then the points in $\phi^{-1}(\beta)$ are $S$-integral relative to $\phi^*(\alpha)$ .  The definition of relative $S$-integrality can be slightly modified so that it behaves well under pullbacks without any restriction on $S$.

Corollary \ref{cor1} can be used to rephrase an integrality statement about backward orbits into an integrality statement about forwards orbits.  However,  some conditions on the Galois orbits of points in $\orb$ will be needed since relative $S$-integrality in $\P^1(\overline{K})$ is defined with respect to all possible embeddings $K\hookrightarrow \overline{K}$.  This rephrasing can give Conjecture \ref{conj} via Silverman's Theorem for $\forb(\beta)$.


\begin{thm}\label{cor2}
For any rational map $\phi$, Conjecture \ref{conj} is true provided there exists an $l$ such that for each $\beta_{l,i}\in\phi^{-l}(\beta)$ and each $m\ge 0$, the points in $\phi^{-m}(\beta_{l,i})$ are all Galois conjugates over $K$.
\end{thm}

\begin{proof}
Enlarge $S$ so that $\phi$ had good reduction at all the places $v\not\in S$, and suppose $\gamma \in \phi^{-n}(\beta)$ is $S$-integral relative to $\alpha$ for $n\ge l$.  Then $\gamma \in \phi^{-m}(\beta_{l,i})$ for some $m\ge 0$, and all the points in $\phi^{-m}(\beta_{l,i})$ are $S$-integral relative to $\alpha$ since they are all Galois conjugates of $\gamma$.  Now $\lambda_{P,v}(\cdot)\ge 0$ and $(\phi^m)^*(\beta_{l,i})$ is an effective divisor (a divisor $\sum n_i P_i$ with each $n_i\ge 0$), so this is equivalent to saying  $(\phi^m)^*(\beta_{l,i})$ is $S$-integral relative to $\alpha$. By Corollary \ref{cor1}, this is furthermore equivalent to having $\phi^m(\alpha)$  being $S$-integral relative to $\beta_{l,i}$.  Therefore, for $n\ge l$, $\gamma \in \phi^{-n}(\beta)$ is $S$-integral relative to $\alpha$ if and only  if there is an $m$ such that $\phi^m(\alpha)$ is $S$-integral relative to $\beta_{l,i}\in \phi^{-l}(\beta)$. Since $\alpha$ is not preperiodic for  $\phi$ and each $\beta_{l,i}$ is not exceptional for $\phi$, Theorem \ref{sil} gives finitely many $m$ for which $\phi^m(\beta)$ is $S$-integral relative  to $\beta_{l,i}$.  Therefore, $$\bigcup_{n\ge l} \phi^{-n}(\beta) $$ contains finitely many points which are $S$-integral relative to $\alpha$.   Altogether, $\orb(\beta)$ contains finitely many points $S$-integral which are relative to $\alpha$.
\end{proof}

If the number if Galois orbits of $\phi^{-n}(\beta)$ is bounded independently of $n$, then the next theorem tells us that hypothesis of the Theorem \ref{cor2} is satisfied.  In Section \ref{dl}, it is show that the Dynamical Lehmer's Conjecture implies such a bound when $\beta$ is not preperiodic for $\phi$, and therefore implies Conjecture \ref{conj} in this case.

\begin{thm}\label{bound}
Suppose the number of Galois orbits for $\phi^{-n}(\beta)$ is bounded by a constant independent of $n$.  Then there exist an $l$ such that for each $\beta_{l,i}\in\phi^{-l}(\beta)$ and each $m\ge 0$, the points in $\phi^{-m}(\beta_{l,i})$ are all Galois conjugates over $K$.
\end{thm}

\begin{proof}
Choose $l$ such that the points in $\phi^{-l}(\beta)$ lie in the maximal number of Galois orbits.  For $\beta_{l,i}\in\phi^{-l}(\beta)$, denote its Galois conjugates over $K$ as $$G(\beta_{l,i}) = \{\sigma(\beta_{l,i}) \mid \sigma \in \mbox{Gal}(\overline K/K) \}.$$

Then $\phi^{-m}(G(\beta_{l,i})) = G(\beta_{l+m,j})$ for some $\beta_{l+m,j}\in \phi^{-l-m}(\beta)$.  Indeed, $G(\beta_{l+m,j}) \subset \phi^{-m}(G(\beta_{l,i}))$ for $\beta_{l+m,j}\in \phi^{-m}(\beta_{l,i})$ since $(\phi^m \circ \sigma)(\beta_{l+m,j})=\sigma(\beta_{l,i})$ for all $\sigma \in \mbox{Gal}(\overline K/K)$.  Therefore each $\phi^{-m}(G(\beta_{l,i}))$ contains at least one Galois orbit over $K$, and by maximality of the number of orbits, each $\phi^{-m}(G(\beta_{l,i}))$ must contain exactly one such orbit.  So $\phi^{-m}(G(\beta_{l,i})) = G(\beta_{l+m,j})$, and the points in $\phi^{-m}(\beta_{l,i})$ are all Galois conjugates over $K$ since $\phi^{-m}(\beta_{l,i})\subset G(\beta_{l+m,j})$.
\end{proof}

R. Jones has established the hypothesis of Theorem \ref{bound} for certain quadratic polynomials with $\beta = 0$ \cite[Prop. 4.5, 4.6, 4.7]{rafe}.   His results, which are summarized in the following proposition, establish Conjecture \ref{conj} for those cases.

\begin{prop}\label{rafe}
Suppose $K=\mathbb Q$ and let $\phi(z)\in \mathbb Z[z]$ be one of the following quadratic polynomial:
\begin{enumerate}
  \item $\phi(z) = z^2 + c$ with $c\not=0,1$
  \item $\phi(z) = z^2 + bz - b$ with $b\not=0$
  \item $\phi(z) = z^2+bz-1$ with $b\not=0$
\end{enumerate}
Then the number of irreducible factors of $\phi^n(z)$ is at most two.
\end{prop}

\begin{cor}
For $\beta =0$, Conjecture \ref{conj} is true for the quadratic polynomials satisfying the hypotheses of Proposition \ref{rafe}.
\end{cor}

\section{Proof of Main Theorem}

In this section, we will prove Theorem \ref{main-thm} by showing there are finitely many $\gamma\in \overline{K}$ satisfying: $\gamma^n=\beta$ for some $n\ge0$ and $\gamma$ is $S$-integral relative to a non-root of unity $\alpha$.   Conjecture \ref{conj} for the map $\phi(z)=z^d$ follows immediately since $\orb(\beta) \subset \{ \gamma \in \overline{K} \mid \gamma^n = \beta, \mbox{ for some } n\in \mathbb Z_{\ge 0} \}$.  Once the result is established for $\phi(z)=z^d$, we may use the functorial properties of integrality prove the conjecture for Chebyshev polynomials.  The main idea of the proof of Theorem \ref{main-thm} involves showing that the Galois orbits for $z^n-\beta$ is uniformly bounded when $\beta$ is not a root of unity.  A more general approach will require an understanding of the Galois group of points for $\phi^{-n}(\beta)$ which can utilize Theorem \ref{cor2}.

For $\beta=0$, Theorem \ref{main-thm} is trivial.  When $\beta$ is a root of unity, it is a theorem of Baker-Ih-Rumely \cite{bir}.

\begin{thm}[Baker-Ih-Rumely]\label{bir}
If $\alpha \in K$ is not 0 or a root of unity, then there are finitely many roots of unity in $\overline K$ which are S-integral relative to $\alpha$.
\end{thm}

Their proof is based on showing that if infinitely many roots of unity $\zeta_n$ are $S$-integral relative to $\alpha$ then the limit $$\lim_{n\to \infty} \frac{1}{[K(\zeta_n):\mathbb Q]}\sum_{v\in M_K} \; \sum_{\sigma:K(\zeta_n)/K \lra \overline{K}_v} \; \log(|\sigma(\zeta_n)-\alpha|_v)$$
converges to $h(\alpha)$, which is nonzero by a theorem of Kronecker \cite[Th. 1.5.9]{hdg}.  This requires a strong equidistribution theorem for the roots of unity and A. Baker's linear forms in logarithm \cite{ABak}.  A contradiction is then obtained by noting, via an interchange of summation and the product formula, that the above limit is zero.

It is possible to adapt their methods to the case when $\beta$ is not 0 or a root of unity; however, Lemma \ref{lang} along with Siegel's Theorem for $\mathbb G_m(K)$ gives a more concise proof.

\begin{lem}\label{lang}  Suppose $\beta \in K$ is not 0 or a root of unity. Then there is a finite extension $L$ of $K$ and a finite subset $D =
\{\beta_1,\dots, \beta_l \} \subset L$ such that every irreducible factor of $z^n-\beta$ over $L$ is of the form $z^m-\beta_i$ with $\beta_i \in D$ and $m\le n$.  Furthermore, each $\beta_i$ is a root of $\beta$.
\end{lem}

\begin{proof}
According to Capelli's Theorem \cite[VI \S 9]{langalg}, $z^n-\beta$ is irreducible over $K$ if the following two conditions holds: $\beta \not\in K^p$ for all $p$ dividing $n$, and $\beta \not\in -4K^4$ when 4 divides $n$.


Assume $4 \not| n$ or $\beta \not\in -4K^4$.   Then $z^n-\beta$ will be reducible over $K$ when $\beta \in K^p$ for a prime $p$ dividing $n$.  When
$\beta$ is not a root of unity, it can only be a prime power in $K$ for finitely many primes.  To see this, note that if $\alpha_i^{p_j}=\beta$ for $\alpha_i \in K$, then the set $\{\alpha_1,\alpha_2,\dots \}$ is a set of bounded height and bounded degree whose cardinality is greater than $\#\{p_1, p_2,\dots \}$.  By Northcott's Theorem, the set $\{p_1, p_2,\dots \}$ is finite.

Suppose $p_t$ is the largest prime for which $\beta$ is a prime power in $K$ and let $L=K(\zeta_p \mid \mbox{primes}\; p\le p_t)$.  Now $\beta \not\in L^q$ for some prime $q>p_t$.  Suppose it were, and note that $X^q - \beta$ would be irreducible over $K$ since $\beta \not\in K^q$.   This means
$[K(\beta^{1/q}):K]=q$ and $q$ would divide $[L:K]=\prod_{p \le p_t}(p-1)$. Therefore $q<p_t$, and this contradicts the assumption that $q>p_t$.  Let
$p_1,\dots,p_l$ be all the primes for which $\beta$ is a prime power in $L$ and let $s_i$ be the largest number such that $\beta \in L^{{p_i}^{s_i}}$.

If $n=p_1^{r_1}m_1$ with $p_1\not| m_1$, then we obtain the following factorization over $L$: $$z^n-\beta=(z^{p_1^{r_1-1}m_1} - a_1)(z^{p_1^{r_1-1}m_1} - \zeta_{p_1} a_1)\cdots (z^{p_1^{r_1-1}m_1} -\zeta_{p_1}^{p_1-1} a_1) $$ where $a_1^{p_1} = \beta$. If $\zeta_{p_1}^{j} a_1$ is a $q$th power in $L$ then $\beta = (\zeta_{p_1}^{j} a_1)^{p_1}$ is also a $q$th power in $L$.  Therefore $\zeta_{p_1}^{j} a_1$ cannot be a prime power in $L$ for any prime $q > p_t$. Furthermore, if $s$ is the largest number for which $\zeta_{p_1}^{j} a_1 \in L^{p_1^s}$ then $(\zeta_{p_1}^{j} a_1)^{p_1} = \beta \in L^{p_1^{s+1}}$, and so $s\le s_1 -1 $.  This means we may continue factoring at most $p_1^{s_1}$ times until we obtain factors of the form $z^{n_j}-b_j$ where $b_j$ is a root of $\beta$ and either $n_j=m_1$ or $b_j \not\in L^{p_1}$.   We now repeat the process for each prime $p_i|n_j$ where $\beta \in L^{p_i}$ and $p_i\not=p_1$. In the end, we will obtain factors $z^m-\beta'$ where $\beta'$ is a root of $\beta$ and either $\beta' \not\in L^{p_i}$ or $p_i \not| m$, for $1\le i\le l$.  Since $4\not|m$, by Capelli's Theorem, we have factored $z^n-\beta$ into irreducible factors.

Suppose $4|n$ and $\beta \in -4K^4$.  Let $\zeta_4=\sqrt{-1}$ and use $-4 = (1\pm \zeta_4)^4$ to note that $\beta \in -4K^4$ implies $\beta \in K(\zeta_4)^4$.  Let $s$ be the largest number for which $b^{4^s}=\beta$ where $b \in K(\zeta_4)$.  If $n=4^rm$ we obtain the following factorization over $K(\zeta_4)$ $$z^n - \beta = (z^{4^{r-1}m}-a_1)(z^{4^{r-1}m}-\zeta_4a_1)(z^{4^{r-1}m}-\zeta_4^2 a_1)(z^{4^{r-1}m}-\zeta_4 ^3 a_1) $$ where $a_1=b^{4^{s-1}}$.  If $\zeta_4^j a_1 \in K(\zeta_4)^4$ we may continue factoring to obtain at most $4^s$ factors of the form $z^{n_j}-b_j$ where $b_j$ is a root of $\beta$ and either $n_j=m$ or $b_j \not\in K(\zeta_4)^4 \supset -4K(\zeta_4)^4$.  This means either $4\not|n_j$ or $b_j \not\in -4K(\zeta_4)^4$, and we have reduced to the initial case where we consider the primes $p|n_j$ for which $b_j\in K(\zeta_4)^p$.  Since $b_j$ being a $q$th power in $K(\zeta_4)$ implies $\beta$ is a $q$th power in $K(\zeta_4)$, we may take $L=K(\zeta_4, \zeta_p \mid \mbox{primes}\; p\le p_t)$ where $p_t$ is the largest prime for which $\beta$ is a prime power in $K(\zeta_4)$.  Repeated factorizations will give irreducible factors of the form $z^m-\beta'$ where $\beta'$ is a root of $\beta$.

We have shown the irreducible factors of $z^n-\beta$ over $L$ are always of form $z^m-\beta'$ where $m\le n$ and $\beta'$ is a root of $\beta$ in $L$.  By Northcott's Theorem, the set $D=\{\beta' \in L \mid \beta' \; \mbox{is a root of $\beta$} \}$ is finite since it is of bounded height and degree.
\end{proof}

We now use Siegel's Theorem for integer points on $\mathbb G_m(K)$ to complete the proof of theorem \ref{main-thm}.

\begin{thm}[Siegel]\label{siegel}
Suppose $\Gamma$ is a finitely generated multiplicative subgroup of $\mathbb G_m(K)$.  Then $\Gamma$ contains finitely many points which are $S$-integral relative to $\alpha \in \mathbb G_m(K)$.
\end{thm}

Siegel's Theorem is usually stated as follows: if a curve $C$ over a number field $K$ has at least three distinct points at infinity then it contains finitely many points with coordinates in $\mathcal O_{K,S}$ \cite[Th.7.3.9]{hdg}.  This is equivalent to saying a curve $C$ contains finitely many points which are $S$-integral relative to three distinct points on $C$ (see \cite{drin}).  Take $\Gamma$, any finitely generated subgroup of $C(K)=\mathbb P^1(K)$ not containing 0 and $\infty$, and extend $S$ to $S'$ so that $\Gamma$ is contained in the set of points which are $S'$-integral relative to $0$ and $\infty$.  Then $\Gamma$ contains finitely many points which are $S'$-integral relative to $\alpha \not= 0, \infty$.  Since extending to $S'$ only increases the number of integral points we get Theorem \ref{siegel}.

\begin{proof}[Proof of Theorem \ref{main-thm}]
Without loss of generality, we may assume that $K$ is large enough for the factorization of Lemma \ref{lang} to hold without further extending $K$.
Indeed, for any finite extension $L$ of $K$, we can take $S_L$ to be the set of primes in $L$ lying over the primes in $S$ and note that the points
$S_L$-integral relative to $\alpha$ contains those points which are $S$-integral relative to $\alpha$.  Therefore, proving the theorem for the larger field $L$ establishes it for the smaller field $K$.

Suppose $\gamma^n=\beta$.   Lemma \ref{lang} implies that $\gamma$ is the root of an irreducible polynomial $z^m-\beta_i$ for some $m\le n$ and some $\beta_i \in \{\beta_1,\dots, \beta_l\}$.  Taking $z=\alpha$ and $K_\gamma$ to be the Galois closure of $K(\gamma)$, the equation $$|\alpha^m - \beta_i|_v = \prod_{\sigma \in \mbox{Gal}(K_{\gamma}/K) }|\alpha - \sigma(\gamma)|_v $$ gives that $\gamma$ is $S$-integral relative to $\alpha$ if and only if there is some $m$ for which $\alpha^m$ is $S$-integral relative to $\beta_i$.  By Siegel's Theorem, there are finitely many points of $\Gamma = \{\alpha^m \mid m \in \mathbb Z \}$ which are $S$-integral relative to each $\beta_i$, for $1\le i\le l$.  And since $\alpha$ is not a root of unity, there are only finitely many $m$ for which $\alpha^m$ is $S$-integral relative to $\beta_1,\dots,\beta_l$.  Therefore, there are finitely many $\gamma$ which are $S$-integral relative to $\alpha$.
\end{proof}

\begin{cor}\label{main-cor}
Conjecture \ref{conj} is true for the map $\phi(z)=z^d$.
\end{cor}

We can now use the projection formula to deduce Conjecture \ref{conj} for \emph{Chebyshev polynomials}.  These are defined as maps $T_d$ making the following diagram commute
$$
\begin{CD}
\mathbb G_m @>z^d>> \mathbb G_m \\
@VV\pi V  @VV\pi V \\
\P^1 @>T_d>> \P^1
\end{CD}
$$
where $\pi$ is a finite morphism. The first few Chebyshev polynomials obtained by taking $\pi(z) = z+z^{-1}$ are 
$$
\begin{array}{lll}
&T_2 = z^2-2 &T_3 = z^3-3z \\
&T_4 = z^4-4z^2+2 &T_5 = z^5-5z^3+5z.
\end{array}
$$
See \cite[Ch. 6]{ADS} for additional information and properties of Chebyshev polynomials.  

\begin{cor}
Conjecture \ref{conj} is true for Chebyshev polynomials.
\end{cor}

\begin{proof}
Extend $S$ to contain all the places of bad reduction for $\pi$, and suppose $\alpha$ is not preperiodic for a Chebyshev polynomial $T_d$.  Since $\pi \circ z^d=T_d \circ \pi$, the points in $\pi^{-1}(\alpha)$ are not preperiodic for $\phi(z)=z^d$.  Let $\mathcal R$ be the set of points of $\mathcal{O}_{T_d}^{-}(\beta)$ which are $S$-integral relative to $\alpha$ and take $\gamma \in \pi^{-1}(\mathcal R)\subset\mathcal{O}_{z_d}^{-}(\pi^{-1}(\beta))$.  Then Corollary $\ref{cor1}$ gives that $\gamma$ is $S$-integral relative to $\pi^*(\alpha)$; consequently, $\gamma$ is $S$-integral relative to each point in $\pi^{-1}(\alpha)$.  Since the points in $\pi^{-1}(\alpha)$ are not preperiodic for $\phi(z)=z^d$, Corollary \ref{main-cor} implies that there are finitely many $\gamma\in\mathcal{O}_{z_d}^{-}(\pi^{-1}(\beta))$ which are $S$-integral relative to the points in  $\pi^{-1}(\alpha)$.  So $\pi^{-1}(\mathcal R)$ is finite, and therefore $\mathcal R$ is also finite.
\end{proof}

\section{Dynamic Lehmer and the Galois orbits of $\phi^{-n}(\beta)$}\label{dl}

For the map $\phi(z)=z^d$, Lemma \ref{lang} implies that when $\beta$ is not preperiodic for $\phi$, the number of Galois orbits of $\phi^n(z)-\beta$ is bounded by constant independent of $n$.  Here we will show that the Dynamical Lehmer's Conjecture implies a similar bound when $\phi$ is any rational map. In view of Theorem \ref{cor2} and Theorem \ref{bound}, this is sufficient to obtain Conjecture \ref{conj} when $\beta$ is not preperiodic for $\phi$.

Suppose $d=\deg(\phi)\ge 2$.  For $\beta \in \mathbb P^1(\overline{K})$, the \emph{canonical height associated to $\phi$} is defined as
$$\hat{h}_{\phi}(\beta) = \lim_{n \to \infty} \frac{h(\phi^n(\beta))}{d^n}.$$ This is due to Silverman and Tate and is useful when studying the dynamics of rational maps.

\begin{prop}
Let $\phi$ have degree $d\ge 2$ and $\hat{h}_{\phi}$ be the canonical height associated to $\phi$.  Then for $\beta \in \mathbb P^1(\overline{K})$,
\begin{enumerate}
  \item $\hat{h}_{\phi}(\beta)=0$ if and only if $\beta$ is preperiodic for $\phi$.
  \item $\hat{h}_{\phi}(\phi(\beta)) = d \hat{h}_{\phi}(\beta)$
  \item $\hat{h}_{\phi}(\beta)=h(\beta) + O(1)$ where $O(1)$ does not depend on $\beta$.
\end{enumerate}
\end{prop}

\begin{proof}
See \cite[Th. 3.20, 3.22]{ADS}.
\end{proof}

Since points with small heights tend to have large degree, it is natural to ask how small we can make $\deg(\alpha)\hat{h}_{\phi}(\alpha)$.  An answer is provided by the Dynamical Lehmer's Conjecture \cite[Conj. 3.25]{ADS}.
\begin{conj}[Dynamical Lehmer]
If $\alpha \in \overline{K}$ is not preperiodic for $\phi$ then there is a constant $C=C(\phi, K)$ not depending on $\alpha$ such that $$\hat{h}_{\phi}(\alpha) > \frac{C}{\deg(\alpha)} $$ where $\deg(\alpha)  = [K(\alpha):K]$.
\end{conj}

For $\phi(z)=z^2$, we have that the canonical height $\hat{h}_{\phi} = h$, the absolute height, and if $K=\mathbb Q$ then the Dynamical Lehmer's Conjecture reduces to the classical Lehmer's Conjecture \cite{lehmer}, which further predicts that $C(z^2, \mathbb Q ) = \log(\Omega)$ where $\Omega=1.1762\dots$ is a root of a 10th degree polynomial.  Much work has been done towards resolving these conjectures.  Currently, the best result for the classical Lehmer's Conjecture is given by Dobrowolski \cite{Dob}: $$h(\alpha)\ge\frac{C}{D(\alpha)}\left(\frac{\log\log D(\alpha)}{\log
D(\alpha)}\right)^3 $$ where $D(\alpha)=[\mathbb Q(\alpha):\mathbb Q]$.  If $\phi$ is a rational map associated to an elliptic curve $E/K$, Masser \cite{Mass} has shown $$\hat{h}_{\phi}(\alpha)\ge \frac{C}{D(\alpha)^3\log^2D(\alpha)}$$ where $D(\alpha)=[K(\alpha):K]$ and $\hat{h}_{\phi}(\alpha)\not=0$.   More recently, a general approach by M. Baker \cite{Mbak}, which involves giving a lower bound for a discriminant sum of Arakelov-Green's functions associated to an arbitrary rational map $\phi$, can be used to obtain Masser's estimate and the bound $h(\alpha)\ge C/(D(\alpha)^2\log D(\alpha))$ for the classical Lehmer's Conjecture.

Write $\phi^n(z)=f_n(z)/g_n(z)$ where $f_n(z)$ and $g_n(z)$ are relatively prime polynomials in $K[z]$.  For $\beta \in K$ define $\mu_{\phi, \beta}(n)$ as the number of irreducible factors of $\phi_{n, \beta}(z)=f_n(z)-\beta g_n(z)$ over $K$.  Since $\phi_{n, \beta}(\gamma)=0$ if and only if $\gamma \in \phi^{-n}(\beta)$, the Galois orbits over $K$ for $\phi^n(z)-\beta$ are grouped according to the irreducible factors of $\phi_{n, \beta}(z)$ over $K$.  In particular, the points in a single Galois orbit are precisely the zeros of the same irreducible factor, and the number of Galois orbits equals the number of irreducible factors.

\begin{thm}\label{irr}
Suppose $\beta \in K$ is not preperiodic for $\phi$.  Then the Dynamical Lehmer's Conjecture implies that $$\mu_{\phi, \beta}(n) \le \frac{\hat{h}_{\phi}(\beta)}{C}$$ where $C=C(\phi, K)$ is the constant in Lehmer's Conjecture.  Consequently, the number of Galois orbits of
$\phi^{-n}(\beta)$ over $K$ is at most $\hat{h}_{\phi}(\beta)/C$.
\end{thm}

\begin{proof}
Write $d=\deg(\phi)$ and suppose $\phi_{n,\beta}(z) = f_n(z)-\beta g_n(z)$ splits into $\m$ irreducible factors over $K$.  Now $\deg(\phi_{n,\beta}(z) ) \le d^n$, so we may take $\gamma \in \phi^{-n}(\beta)$ such that $\deg(\gamma) \le d^n/\m$.  Since $\hat{h}_{\phi}(\beta)=d^n\hat{h}_{\phi}(\gamma)$, the Dynamical Lehmer's Conjecture gives $$\frac{\hat{h}_{\phi}(\beta)}{d^n}=\hat{h}_{\phi}(\gamma) \ge \frac{C}{\deg(\gamma)}\ge \frac{C}{d^n} \m.$$  This  implies $\m \le \hat{h}_{\phi}(\beta)/C$.
\end{proof}

\bibliographystyle{amsalpha}
\bibliography{back_orb}

\providecommand{\bysame}{\leavevmode\hbox to3em{\hrulefill}\thinspace}
\providecommand{\MR}{\relax\ifhmode\unskip\space\fi MR }
\providecommand{\MRhref}[2]{%
  \href{http://www.ams.org/mathscinet-getitem?mr=#1}{#2}
}
\providecommand{\href}[2]{#2}
\begin{thebibliography}{FRL06}

\bibitem[Bak75]{ABak}
A.~Baker, \emph{Transcendental number theory}, Cambridge University Press,
  Cambridge, 1975.

\bibitem[Bak06]{Mbak}
M.~Baker, \emph{A lower bound for average values of dynamical {G}reen's
  functions}, Math. Res. Lett. \textbf{13} (2006), no.~2-3, 245--257.

\bibitem[BG06]{hdg}
E.~Bombieri and W.~Guber, \emph{Heights in diophantine geometry}, Cambridge
  University Press, Cambridge, 2006.

\bibitem[BIR08]{bir}
M.~Baker, S.~Ih, and R.~Rumley, \emph{A finiteness property of torsion points},
  Algebra Number Theory \textbf{2} (2008), no.~2, 217--248.

\bibitem[BR06]{br}
M.~Baker and R.~Rumely, \emph{Equidistribution of small points, rational
  dynamics, and potential theory}, Ann. Inst. Fourier (Grenoble) \textbf{56}
  (2006), no.~3, 625--688.

\bibitem[CL06]{CL}
A.~Chambert-Loir, \emph{Mesures et {\'e}quidistribution sur les espaces de
  {B}erkovich}, J. Reine Angew. Math \textbf{595} (2006), 215--235.

\bibitem[Dob79]{Dob}
E.~Dobrowolski, \emph{On a question of {L}ehmer and the number of irreducible
  factors of a polynomial}, Acta Arith \textbf{34} (1979), no.~4, 391--401.

\bibitem[FRL06]{FRL}
C.~Favre and J.~Rivera-Letelier, \emph{Equidistribution quantitative des points
  de petite hauteur sur la droite projective}, Math. Ann. \textbf{335} (2006),
  no.~2, 311--361.

\bibitem[GT08]{drin}
D.~Ghioca and T.~J. Tucker, \emph{Equidistribution and integral points for
  drinfeld modules}, Trans. Amer. Math. Soc. \textbf{360} (2008), no.~9,
  4863--4887.

\bibitem[Jon08]{rafe}
R.~Jones, \emph{The density of prime divisors in the arithmetic dynamics of
  quadratic polynomials}, J. Lond. Math. Soc. (2) \textbf{78} (2008), no.~2,
  523--544.

\bibitem[Lan02]{langalg}
S.~Lang, \emph{Algebra, revised 3rd ed.}, GTM 211, Springer-Verlag, New York,
  2002.

\bibitem[Leh33]{lehmer}
D.~H. Lehmer, \emph{Factorization of certain cyclotomic functions}, Ann. of
  Math. (2) \textbf{34} (1933), no.~3, 461--479.

\bibitem[Lyu83]{lyu}
M.~Lyubich, \emph{Entropy properties of rational endomorphisms of the {R}iemann
  sphere}, Ergodic Theory Dynam. Systems (1983), no.~3, 351--385.

\bibitem[Mas89]{Mass}
D.~W. Masser, \emph{Counting points of small height on elliptic curves}, Bull.
  Soc. Math. France \textbf{117} (1989), no.~2, 247--265.

\bibitem[Pet08]{clay}
C.~Petsche, \emph{S-integral preperiodic points for dynamical systems over
  number fields}, Bull. Lond. Math. Soc. \textbf{40} (2008), no.~5, 749--758.

\bibitem[Sil93]{sil}
J.~H. Silverman, \emph{Integer points, {D}iophantine approximation, and
  iteration of rational maps}, Duke Math. J. \textbf{71} (1993), no.~3,
  793--829.

\bibitem[Sil07]{ADS}
\bysame, \emph{The arithmetic of dynamical systems}, Springer, New York, 2007.

\end{thebibliography}

\end{document}